\documentclass[12pt,a4paper]{amsart}
\usepackage{amsmath,amssymb,latexsym,tabularx}
\usepackage[dvips]{graphicx}
\headsep=1.2500cm \numberwithin{equation}{section}
\newtheorem{theorem}{Theorem}[section]
\newtheorem{proposition}[theorem]{Proposition}
\newtheorem{corollary}[theorem]{Corollary}

\theoremstyle{remark}
\newtheorem{remark}[theorem]{Remark}
\newtheorem{definition}[theorem]{Definition}
\newtheorem{example}[theorem]{Example}

\title[ Asymptotic resemblance relations on Groups]{ Asymptotic resemblance relations on Groups}

\author[Sh. Kalantari]{Sh. Kalantari}
 \address[Sh. Kalantari]{Department of Basic Sciences, Babol Noshirvani University of Technology, Shariati Ave.,Babol, Iran, Post Code:47148-71167.}
 \email{shahab.kalantari@nit.ac.ir}


\begin{document}
\maketitle
\begin{abstract}
In this paper, we study properties of asymptotic resemblance relations induced by compatible coarse structures on groups. We generalize the notion of asymptotic dimensiongrad for groups with compatible coarse structures and show this notion is coarse invariant. We end by defining the notion of set theoretic coupling for groups with compatible coarse structures and showing this notion is the generalization of the notion of topological coupling for finitely generated groups. We show if two groups with compatible coarse structures admit a set theoretic coupling then they are asymptotic equivalent.

\end{abstract}
\section{Introduction}
The Gromov's approach of investigating finitely generated groups with large scale properties of their Cayley graphs (\cite{gro}) has been generalized rapidly for studying more general objects. The key tool for this purpose is the notion of \emph{coarse structure} introduced by Roe (\cite{Roe}). A coarse structure on a set $X$ is a family of subsets of $X\times X$ with some additional properties. Coarse structures have opened a new perspective for investigating more general forms of groups, for instance, \emph{topological groups}. There have been some attempts for this goal. For example in \cite{pr1}, the authors defined the large scale counterpart of topological groups. In fact in \cite{pr1} they defined compatible \emph{ball structures} on groups and for this reason they used the notion of \emph{group ideals}. Ball structures can be considered as alternative ways for defining coarse structures on sets. In \cite{Nicas} a way of defining coarse structures on general groups has been introduced. They used the notion of \emph{generating family} on a group $G$ for defining compatible coarse structures with the group structure of $G$. The notion of generating family and the notion of group ideal are closely related and they are equivalent in some ways. There are also other works in this area, one can see \cite{pr2,pr4,pr3} and the relatively new monograph \cite{Ros}. We survey elementary properties of coarse structures on nonempty sets and compatible coarse structures on groups in \S \ref{1.1}. As another way of defining large scale structures on a given set, the authors of \cite{me} introduced the notion of \emph{asymptotic resemblance relations} (AS.R). An asymptotic resemblance relation on a set $X$ is an equivalence relation on the family of all subsets of $X$ with two additional properties. It can be shown that every coarse structure on a set $X$ can induce an asymptotic resemblance relation on $X$ and almost all concepts in coarse geometry can be carried on to asymptotic resemblance spaces. In \S \ref{1.2} we discuss asymptotic resemblance relations and their basic properties briefly.\\
In Section \ref{sec2} we describe asymptotic resemblance relations on groups induced by compatible coarse structures. In this section, we drive special meanings of concepts in AS.R spaces for the AS.R relations induced by compatible coarse structures on groups. In particular, we show what are the notions of \emph{asymptotic disjointness} and \emph{compatible AS.R with the topology} equivalent to.\\
Section \ref{sec3} is devoted to generalize the concept of \emph{asymptotic dimensiongrad} for groups with compatible coarse structures. The notion of asymptotic dimensiongrad for metric spaces has been introduced in \cite{asdim} as an inductive way for finding dimensions in large scale. In this section, we describe the meaning of an \emph{asymptotic cut} between two asymptotic disjoint subsets of a coarse space and define the asymptotic dimensiongrad inductively by using this notion. In addition, we show asymptotic dimensiongrad is invariant under coarse equivalences and we study the relation between asymptotic cut and \emph{large scale separator}. Even though we define the asymptotic dimensiongrad for all coarse spaces, our main subject in this section is the asymptotic dimensiongrad of groups with compatible coarse structures.\\
The notion of \emph{topological coupling} has first appeared in \cite{gro}. A topological coupling for two finitely generated groups $G$ and $H$ is a locally compact topological space $X$ such that $G$ and $H$ act cocompactly and properly on $X$ from left and right respectively and these actions commute. It can be shown that two finitely generated groups are coarsely equivalent if and only if they admit a topological coupling (\S 0.2.$C_{2}^{\prime}$ of \cite{gro}). In Section \ref{sec4} we show how the notion of topological coupling can be generalized to the concept of \emph{set theoretic coupling} for groups with AS.R relations induced by compatible coarse structures. We show that if two groups admit a set theoretic coupling then they are asymptotic equivalent. This section contains some examples of set theoretic couplings for some important compatible coarse structures on groups.
\section{preliminaries}
\subsection{Coarse Structures}\label{1.1}
We start by recalling that a \emph{coarse structure} $\mathcal{E}$ on a nonempty set $X$ is a family of subsets of $X\times X$ such that each subset of an element of $\mathcal{E}$ is an element of $\mathcal{E}$ and for each $E,F\in \mathcal{E}$ the subsets $E\bigcup F$, $E^{-1}$ and $E\circ F$ belong to $\mathcal{E}$, where $E^{-1}=\{(y,x)\in X\times X\mid (x,y)\in E\}$ and $E\circ F$ denotes the set of all $(x,y)\in X\times X$ such that $(x,z)\in F$ and $(z,y)\in E$ for some $z\in X$ (\cite{Roe}). If $\mathcal{E}$ is a coarse structure on the set $X$, then the pair $(X,\mathcal{E})$ is called a coarse space. The coarse structure $\mathcal{E}$ on the set $X$ is called unitary if it contains the diagonal $\Delta=\{(x,x)\mid x\in X\}$ and it is called connected if for each $x,y\in X$, $\{(x,y)\}\in \mathcal{E}$. In the rest of this paper by a coarse structure we mean a unitary and connected coarse structure. Before going further we should recall that a subset $B$ of a coarse space $(X,\mathcal{E})$ is called bounded if $B\times B\in \mathcal{E}$. Suppose that $(X,\mathcal{E})$ is a coarse space and let $Y\subseteq X$ be a nonempty set. Define $\mathcal{E}_{Y}=\{E\bigcap (Y\times Y)\mid E\in \mathcal{E}\}$. Then $\mathcal{E}_{Y}$ is a coarse structure on $Y$, which is called the coarse structure induced on $Y$ by $\mathcal{E}$.
\begin{example}
The metric coarse structure $\mathcal{E}_{d}$ on a metric space $(X,d)$ is the family of all subsets $E$ of $X\times X$ such that for some $k>0$, $d(x,y)\leq k$ for all $(x,y)\in E$.
\end{example}
In the category of coarse spaces morphisms are called \emph{coarse maps}.
\begin{definition}
Let $(X,\mathcal{E})$ and $(Y,\mathcal{E}^{\prime})$ be two coarse spaces. A map $f:X\rightarrow Y$ is called a \emph{coarse map} if for each $E\in \mathcal{E}$, $(f\times f)(E)\in \mathcal{E}^{\prime}$ and the inverse image of each bounded subset of $Y$ is a bounded subset of $X$. A coarse map $f:X\rightarrow Y$ is called a \emph{coarse equivalence} if there exists a coarse map $g:Y\rightarrow X$ such that $\{(g(f(x)),x)\mid x\in X\}\in \mathcal{E}$ and $\{(f(g(y)),y)\mid y\in Y\}\in \mathcal{E}^{\prime}$. In this case we call $g$ a \emph{coarse inverse} of $f$.
\end{definition}
The following definition is from \cite{Nicas}.
\begin{definition}
Let $G$ be a group. For subsets $A$ and $B$ of $G$ suppose that $A^{-1}=\{g^{-1}\mid g\in G\}$ and $AB=\{gh\mid g\in A,h\in B\}$. By a \emph{generating family} $\mathcal{F}$ on the group $G$ we mean a family $\mathcal{F}$ of subsets of $G$ such that $\mathcal{F}$ contains a nonempty subset of $G$ and for each $A,B\in \mathcal{F}$, the subsets $A\bigcup B$, $AB$ and $A^{-1}$ belongs to $\mathcal{F}$.
\end{definition}
\begin{example}
Let $G$ be a group and let $\mathcal{F}$ be a generating family on $G$. Let $\mathcal{E}_{\mathcal{F}}$ denotes the family of all subsets $E$ of $X\times X$ such that $E\subseteq G(F\times F)$ for some $F\in \mathcal{F}$, where
$$G(F\times F)=\{(ga,gb)\mid g\in G,a,b\in F\}$$
It can be shown that the family $\mathcal{E}_{\mathcal{F}}$ is a coarse structure on $G$ and we call it the coarse structure associated to the generating family $\mathcal{F}$ on $G$ (Proposition 2.4 of \cite{Nicas}). This coarse structure is not connected in general. In fact the coarse structure $\mathcal{E}_{\mathcal{F}}$ is connected if and only if $G=\bigcup_{F\in \mathcal{F}}F$ (Proposition 2.21 of \cite{Nicas}). From now on we assume all generating sets $\mathcal{F}$ with the additional assumption $G=\bigcup_{F\in \mathcal{F}}F$.
\end{example}
Let $G$ be a group and let $\mathcal{F}$ be a generating family on $G$. We denote by $\hat{\mathcal{F}}$ the family of all $A\subseteq G$ such that $A\subseteq B$ for some $B\in \mathcal{F}$. It is easy to see that $\hat{\mathcal{F}}$ is a generating family on $G$ and $\mathcal{E}_{\mathcal{F}}=\mathcal{E}_{\hat{\mathcal{F}}}$. Another thing that we should mention here is that if the coarse structure $\mathcal{E}_{\mathcal{F}}$ is connected then the family of all bounded subsets coincide with the family $\hat{\mathcal{F}}$. Let us end this section with two examples of generating sets on groups (for more examples see \cite{Nicas}).
\begin{example}
Let $G$ be a group and suppose that $\mathcal{F}$ denotes the family of all finite subsets of $G$. Then $\mathcal{F}$ is a generating family on $G$.
\end{example}
\begin{example}
Let $G$ be a (Hausdorff) topological group. Let $\mathcal{K}$ denotes the family of all subsets of $G$ with compact closure. Then $\mathcal{K}$ is a generating family on $G$. We call the associated coarse structure $\mathcal{E}_{\mathcal{K}}$ the \emph{group compact coarse structure} (\cite{Nicas}).
\end{example}
\begin{example}
Let $G$ be a group and let $d$ be a left invariant metric on $G$. Assume that $\mathcal{F}$ denotes the family of all bounded subsets of the metric space $(G,d)$. The family $\mathcal{F}$ is a generating family on $G$.
\end{example}
Let $X$ be a topological space. A coarse structure $\mathcal{E}$ on $X$ is called to be compatible with the topology of $X$ if there exists some $E\in \mathcal{E}$ which is open in the product topology of $X\times X$ and $E$ contains the diagonal. The coarse structure $\mathcal{E}$ is called to be proper if each bounded subset of the coarse space $(X,\mathcal{E})$ has a compact closure. Note that if there exists a compatible and proper coarse structure on a topological space $X$ then $X$ is locally compact. Now let $\mathcal{E}$ be a compatible and proper coarse structure on the topological space $X$. A bounded and continuous function $f:X\rightarrow \mathbb{C}$ is called a \emph{Higson function} if for each $E\in \mathcal{E}$ the restriction of the function $g(x,y)=\mid f(x)-f(y)\mid$ to $E$ vanishes to infinity. If we denote the family of all Higson functions by $C_{h}(X)$ then there exists a compactification $hX$ of $X$ such that $C_{h}(X)=C(hX)$ (\cite{Roe} section 2.3), where $C(hX)$ denotes the family of all continuous functions $f:hX\rightarrow \mathbb{C}$. The compactification $hX$ of $X$ is called the \emph{Higson compactification} of $X$. The boundary $\nu X=hX\setminus X$ is called the \emph{Higson corona} of $X$. For a subset $A$ of $X$ suppose that $\bar{A}$ denotes the closure of $A$ in $hX$. We denote the boundary $\bar{A}\cap \nu X$ by $\nu A$.
\subsection{Asymptotic Resemblance}\label{1.2}
Let $X$ be a nonempty set by an \emph{asymptotic resemblance relation} (an AS.R) on $X$ we mean an equivalence relation $\lambda$ on the family of all subsets of $X$ with the following two additional properties,\\
1) If $A_{1},A_{2},B_{1},B_{2}\subseteq X$ and $A_{1}\lambda B_{1}$ and $A_{2}\lambda B_{2}$ then $(A_{1}\cup A_{2})\lambda (B_{1}\cup B_{2})$.\\
2) Suppose that $A_{1},A_{2},B\subseteq X$ and $A_{1},A_{2}\neq \emptyset$. If $(A_{1}\cup A_{2})\lambda B$ then there exists nonempty subsets $B_{1}$ and $B_{2}$ of $X$ such that $B=B_{1}\cup B_{2}$ and $A_{i}\lambda B_{i}$ for $i\in \{1,2\}$.\\
If $\lambda$ denotes an AS.R on the set $X$ then we call the pair $(X,\lambda)$ an AS.R space. If $A\lambda B$ for two subsets $A$ and $B$ of the AS.R space $(X,\lambda)$, we say $A$ and $B$ are asymptotically alike. One can see \cite{me} for more details about AS.R spaces. Let $Y$ be a nonempty subset of the AS.R space $(X,\lambda)$. For two subsets $A$ and $B$ of $Y$ define $A\lambda_{Y}B$ if $A$ and $B$ are two asymptotically alike subsets of $X$. The relation $\lambda_{Y}$ is an AS.R on $Y$, which we call the subspace AS.R induced on $Y$ from $X$.
\begin{example}
Let $(X,d)$ be a metric space. For two subsets $A$ and $B$ define $A\lambda_{d}B$ if $d_{H}(A,B)<+\infty$, where $d_{H}$ denotes the Hausdorff distance between $A$ and $B$. Then $\lambda_{d}$ is an AS.R on $X$, which we call the metric AS.R associated to the metric $d$ on $X$.
\end{example}
The following example illustrates how the notions of coarse structures and asymptotic resemblance relations are related.
\begin{example}
Let $\mathcal{E}$ be a coarse structure on the set $X$. Suppose that for $E\in \mathcal{E}$ and $x\in X$, $E(x)=\{y\in X\mid (y,x)\in E\}$ and for $A\subseteq X$, $E(A)=\bigcup_{x\in A}E(x)$. For two subsets $A$ and $B$ of $X$ define $A\lambda_{\mathcal{E}}B$ if $A\subseteq E(B)$ and $B\subseteq E(A)$ for some $E\in \mathcal{E}$. The relation $\lambda_{\mathcal{E}}$ is an AS.R on $X$ and we call it the AS.R associated to the coarse structure $\mathcal{E}$ on $X$. It is worth mentioning that two different coarse structures on a set $X$ may have the same associated AS.R (Example 3.1 of \cite{me}).
\end{example}
We call the subset $A$ of the AS.R space $(X,\lambda)$ bounded if $A=\emptyset$ or $A\lambda \{x\}$ for some $x\in X$. If $\mathcal{E}$ is a coarse structure on the set $X$ then $A\subseteq X$ is bounded in the coarse space $(X,\mathcal{E})$ if and only if $A$ is bounded in the AS.R space $(X,\lambda_{\mathcal{E}})$.\\ Suppose that $(X,\lambda)$ and $(Y,\lambda^{\prime})$ are two AS.R spaces. A map $f:X\rightarrow Y$ is called an \emph{AS.R mapping} if the inverse image of each bounded subset of $Y$ is a bounded subset of $X$ and if $A\lambda B$ then $f(A)\lambda^{\prime}f(B)$, for all $A,B\subseteq X$. An AS.R mapping $f:X\rightarrow Y$ is called an \emph{asymptotic equivalence} if there exists an AS.R mapping $\tilde{f}:Y\rightarrow X$ such that $\tilde{f}(f(A))\lambda A$ and $f(\tilde{f}(B))\lambda^{\prime}B$ for all $A\subseteq X$ and $B\subseteq Y$. In this case $\tilde{f}$ is called the asymptotic inverse of $f$ and $X$ and $Y$ are called asymptotically equivalent.
\begin{definition}
We say two subsets $A$ and $B$ of the AS.R space $(X,\lambda)$ are \emph{asymptotically disjoint} if $L_{1}\lambda L_{2}$ for $L_{1}\subseteq A$ and $L_{2}\subseteq B$ then $L_{1}$ and $L_{2}$ are bounded. In other words two subsets of an AS.R space are asymptotically disjoint if they do not have any unbounded asymptotically alike subsets. An AS.R space $(X,\lambda)$ is called \emph{asymptotically normal} if $A$ and $B$ are two asymptotically disjoint subsets of $X$ then $X=X_{1}\cup X_{2}$ such that $X_{1}$ and $X_{2}$ are asymptotically disjoint from $A$ respectively.
\end{definition}
\begin{proposition}\label{normal}
Let $d$ be a metric on the set $X$. Then the AS.R space $(X,\lambda_{d})$ is asymptotically normal.
\end{proposition}
\begin{proof}
See Proposition 4.5 of \cite{me}.
\end{proof}
Let $\lambda$ be an AS.R on the topological space $X$. We say $\lambda$ is compatible with the topology of $X$ if for each subset $A$ of $X$ there exists an open subset $A\subseteq U$ such that $A\lambda U$ and in addition $\bar{A}\lambda A$. The AS.R $\lambda$ is called to be proper if every bounded subset has compact closure. It can be shown that each AS.R associated to a compatible and proper coarse structure on a topological space is compatible and proper (Proposition 4.2 of \cite{me}). Let $\lambda$ be a compatible and proper AS.R on the normal topological space $X$. Assume that $(X,\lambda)$ is asymptotically normal. Define $A\delta B$ if $A$ and $B$ are not asymptotically disjoint or $\bar{A}\cap \bar{B}\neq \emptyset$. The relation $\delta$ defines a compatible proximity on $X$ and we call the associated Smirnov compactification the asymptotic compactification of $X$ (for proximity spaces and their Smirnov compactifications see \cite{Nai}). It can be shown that if $\mathcal{E}$ is a compatible and proper coarse structure on the normal topological space $X$, such that $\lambda_{\mathcal{E}}$ is asymptotically normal then the Higson compactification of $X$ and the asymptotic compactification of $X$ are equal up to a homeomorphism (\cite{me}).
\section{asymptotic resemblance relations on groups}\label{sec2}
\begin{definition}
Let $\mathcal{F}$ be a generating family on a group $G$. Define $A\lambda_{\mathcal{F}} B$ if there exists some $H\in \mathcal{F}$ such that $A\subseteq BH$ and $B\subseteq AH$.
\end{definition}
\begin{proposition}
Let $\mathcal{F}$ be a generating family on the group $G$. Then $\lambda_{\mathcal{F}}$ is the AS.R associated to the coarse structure $\mathcal{E}_{\mathcal{F}}$ on $G$.
\end{proposition}
\begin{proof}
Assume that $\lambda$ denotes the AS.R associated to $\mathcal{E}_{\mathcal{F}}$. For two subsets $A$ and $B$ of $X$ if $A\lambda B$, then there exists some $H\in \mathcal{F}$ and $E\subseteq G(H\times H)$ such that $A\subseteq E(B)$ and $B\subseteq E(A)$. Let $K=H^{-1}H$. So $K\in \mathcal{F}$. If $a\in A$, there is some $b\in B$ such that $(a,b)\in E$. Thus there are $h_{1},h_{2}\in H$ and $g\in G$ such that $(a,b)=(gh_{1},gh_{2})$. So $b^{-1}a=h_{2}^{-1}h_{1}\in K$ and it shows that $a\in BK$ and therefore $A\subseteq BK$. Similarly one can show that $B\subseteq AK$ and this leads to $A\lambda_{\mathcal{F}}B$. Now suppose that $A\lambda_{\mathcal{F}}B$. There exists some $H\in \mathcal{F}$ such that $A\subseteq BH$ and $B\subseteq AH$. Let $K=H\cup HH^{-1}$ and $E=G(K\times K)$. Clearly $E\in \mathcal{F}$. If $a\in A$, then there exists some $b\in B$ and $h\in H$ such that $a=bh$. Since $e\in HH^{-1}\subseteq K$, we have $(a,b)=b(h,e)\in E$ ($e$ denotes the neutral element of $G$). So $A\subseteq E(B)$ and similarly it can be shown that $B\subseteq E(A)$. Therefore $A\lambda B$.
\end{proof}
\begin{proposition}\label{disjoint}
Let $\mathcal{F}$ be a generating family on the group G. Then two subsets $A$ and $B$ of $G$ are asymptotically disjoint in the AS.R space $(G,\lambda_{\mathcal{F}})$ if and only if $A\cap BH$ is bounded for each $H\in \mathcal{F}$.
\end{proposition}
\begin{proof}
Assume that $A$ and $B$ are two asymptotically disjoint subsets of $(G,\lambda_{\mathcal{F}})$. Suppose that $L_{1}=A\cap BH$ is unbounded for some $H\in \mathcal{F}$. Let $L_{2}=B\cap AH^{-1}$ and $K=H\cup H^{-1}$. It is straightforward to show that $L_{1}\subseteq L_{2}K$ and $L_{2}\subseteq L_{1}K$, thus $L_{1}\lambda_{\mathcal{F}}L_{2}$. It shows that $A$ and $B$ are not asymptotically disjoint, a contradiction. To prove the converse suppose that $A$ and $B$ are two subsets of $G$ and they are not asymptotically disjoint. So there are unbounded subsets $L_{1}\subseteq A$ and $L_{2}\subseteq B$ and some $H\in \mathcal{F}$ such that $L_{1}\subseteq L_{2}H$ and $L_{2}\subseteq L_{1}H$. Thus $L_{1}\subseteq A\cap BH$, which shows that $A\cap BH$ is unbounded.
\end{proof}
Since in this paper we consider each coarse structure to be connected so as we mentioned before if $\mathcal{F}$ is a generating family on the group $G$, then the family of all bounded subsets of $(G,\lambda_{\mathcal{F}})$ coincide with $\hat{\mathcal{F}}$. Thus we have the following corollary.
\begin{corollary}\label{dis}
Let $\mathcal{F}$ be a generating family on the group $G$. Two subsets $A$ and $B$ of $G$ are asymptotically disjoint if and only if $A\cap BH\in \hat{\mathcal{F}}$, for all $H\in \mathcal{F}$.
\end{corollary}
\begin{corollary}
Let $G$ be a topological group. Then two subsets $A$ and $B$ of $G$ are asymptotically disjoint with respect to the AS.R $\lambda_{\mathcal{K}}$ if and only if for all compact subset $K$ of $G$, $A\cap BK$ has compact closure ($\mathcal{K}$ denotes the family of all compact subsets of $G$).
\end{corollary}
\begin{proof}
It is evident from Corollary \ref{dis}.
\end{proof}
\begin{proposition}\label{compatible}
Let $\mathcal{F}$ be a generating family on the topological group $G$. Then the coarse structure $\mathcal{E}_{\mathcal{F}}$ is compatible with the topology of $G$ if and only if there exists an open set $U$ containing the identity such that $U\in \hat{\mathcal{F}}$.
\end{proposition}
\begin{proof}
Let $e$ denotes the identity of $G$. Suppose that $\mathcal{E}_{\mathcal{F}}$ is compatible with the topology of $G$ and let $E$ be an open entourage that contains the diagonal. Then $e\in E(e)$ is an open set and $E(e)\in \hat{\mathcal{F}}$. To prove the converse suppose that $U$ is an open neighbourhood of the identity and $U\in \hat{\mathcal{F}}$. Then clearly $E=G(U\times U)$ is an open entourage containing the diagonal.
\end{proof}
\begin{proposition}\label{proper}
Let $\mathcal{F}$ be a generating family on the topological group $G$. Then the coarse structure $\mathcal{E}_{\mathcal{F}}$ is proper if and only if each element of $\mathcal{F}$ has compact closure.
\end{proposition}
\begin{proof}
It is straightforward.
\end{proof}
On a topological group $G$, we call the AS.R $\lambda_{\mathcal{K}}$ the \emph{group compact AS.R} on $G$ ($\mathcal{K}$ denotes the family of all subsets of $G$ with compact closure as mentioned before). By Proposition \ref{proper} it is evident that the group compact AS.R is proper. The following corollary is a direct consequence of Proposition \ref{compatible}.
\begin{corollary}\label{com}
The group compact AS.R is compatible with the topology of the topological group $G$ if and only if $G$ is locally compact.
\end{corollary}
Let $\lambda_{\mathcal{F}}$ denotes the AS.R on the topological group $G$ associated to the generating set $\mathcal{F}$ on $G$. For two subsets $A$ and $B$ of $G$ define $A\delta_{\mathcal{F}}B$ if $\bar{A}\cap \bar{B}\neq \emptyset$ or there exists some $K\in \mathcal{F}$ such that $A\cap BK$ is not in $\hat{F}$ (i.e they are not asymptotically disjoint). If $\lambda_{\mathcal{F}}$ is an asymptotically normal, proper and compatible AS.R on the normal topological group $G$ then $\delta_{\mathcal{F}}$ is a compatible proximity on $G$. Thus the Smirnov compactification of the proximity space $(G,\delta_{\mathcal{F}})$ is homeomorphic with the Higson compactification of $(G,\mathcal{E}_{\mathcal{F}})$.\\
Unfortunately the AS.R $\lambda_{\mathcal{F}}$ fails to be asymptotically normal in general. In the following we give an example of a locally compact normal topological group such that the group compact AS.R is not asymptotically normal.
\begin{example}
Suppose the additive group $\mathbb{R}$ with the discrete topology. Note that in this topological group the family $\mathcal{K}$ is equal to the family of all finite subsets of $\mathbb{R}$. We claim that AS.R $\lambda_{\mathcal{K}}$ on this topological group is not asymptotically normal. To prove this it suffices two show that there are two subsets $A$ and $B$ of $\mathbb{R}$ such that they are asymptotically disjoint but for each asymptotically disjoint subset $X$ from $A$, the subset $\mathbb{R}\setminus X$ is not asymptotically disjoint from $B$. Let $A=[0,+\infty)$ and $B=\{-n\mid n\in \mathbb{N}\}$. Two subsets $A$ and $B$ of $\mathbb{R}$ are asymptotically disjoint. Suppose that $X\subseteq \mathbb{R}$ is asymptotically disjoint from $A$. So for each $n\in \mathbb{N}$ the set $A_{n}=\{\epsilon >0\mid -n+\epsilon \in X\}$ is finite. So $Y=\bigcup_{n\in \mathbb{N}}A_{n}$ is countable. Thus there exists some $\epsilon >0$ such that for each $n\in \mathbb{N}$, $-n+\epsilon \notin X$. This means $\mathbb{R}\setminus X$ is not asymptotically disjoint from $B$.
\end{example}
Let us recall that a Hausdorff topological space $X$ is called \emph{hemicompact} if there exists a sequence $K_{1}\subseteq K_{2}\subseteq ...$ of compact subsets of $X$ such that $X=\bigcup_{i\in \mathbb{N}} K_{i}$ and each compact subset of $X$ contains in $K_{i}$ for some $i\in \mathbb{N}$. It can be shown that the group compact coarse structure on a topological group $G$ coincides with the metric coarse structure for some left invariant metric on $G$ if and only if $G$ is hemicompact (It is a straightforward consequence of Theorem 2.17 of \cite{Nicas}). Since each metrizable AS.R is asymptotically normal (Proposition \ref{normal}) we have the following result.
\begin{proposition}
Let $G$ be a hemicompact topological group. Then the group compact AS.R is asymptotically normal.
\end{proposition}
\section{Asymptotic dimensiongrad}\label{sec3}
Suppose that $(X,d)$ is a metric space and $r>0$. Let us recall that a $r$-chain from $x\in X$ to $y\in X$ is a finite sequence $x=x_{0},...,x_{n}=y$ such that $d(x_{i},x_{i+1})\leq r$ for all $i\in\{0,...,n-1\}$.
\begin{definition}\label{chain1}
Let $(X,\mathcal{E})$ be a coarse space and $E\in \mathcal{E}$. A $E$-chain in $(X,\mathcal{E})$ from $a\in X$ to $b\in X$ is a finite sequence $a=x_{0},...,b=x_{n}$ such that $(x_{i},x_{i+1})\in E$, for all $i\in\{0,...,n-1\}$.
\end{definition}
\begin{definition}\label{chian2}
Suppose that $G$ is a group and $\mathcal{F}$ is a generating family on $G$. Let $X\subseteq G$ be a nonempty set. For $K\in \mathcal{F}$ by a $K$-chain in $X$ from $x\in X$ to $y\in X$ we mean the finite sequence $x=g_{0},...,g_{n}=y$, such that for each $i\in \{1,...,n\}$, $g_{i}\in X$ and $g_{i}\in g_{i-1}K$.
\end{definition}
The following proposition is a direct consequence of Definition \ref{chain1} and Definition \ref{chian2}.
\begin{proposition}
Let $\mathcal{F}$ be a generating family on the group $G$. Then $a=g_{0},...,g_{n}=b$ is a $E$-chain in $(X,\mathcal{E}_{\mathcal{F}})$ for some $E\in \mathcal{E}_{\mathcal{F}}$ if and only if it is a $K$-chain for some $K\in \mathcal{F}$.
\end{proposition}
\begin{proposition}\label{metcut}
Suppose that $d$ is a left invariant metric on the group $G$. Let $\mathcal{F}$ denote the family of all bounded subsets of $(G,d)$ and let $X$ be a nonempty subset of $G$. Then the sequence $x=g_{0},...,g_{n}=y$ is a $r$-chain in $X$ for some $r>0$ if and only if there exists some $K\in \mathcal{F}$ such that $x=g_{0},...,g_{n}=y$ is a $K$-chain between $x$ and $y$ in $X$.
\end{proposition}
\begin{proof}
It is straightforward.
\end{proof}
\begin{definition}\label{cut}
Let $(X,\mathcal{E})$ be a coarse space. Suppose that $A$ and $B$ are two asymptotically disjoint subsets of $X$. We call $C\subseteq X$ an \emph{asymptotic cut} between $A$ and $B$ if $C$ is asymptotically disjoint from both $A$ and $B$ and for each $E\in \mathcal{E}$ there exists some $F\in \mathcal{E}$ such that each $E$-chain from an element of $A$ to an element of $B$ meets $F(C)$.
\end{definition}
\begin{proposition}\label{cut1}
Let $G$ be a group and let $\mathcal{F}$ denote a generating family on $G$. Assume that $X\subseteq G$ is nonempty and consider $X$ with the subspace coarse structure. For two asymptotically disjoint subsets $A$ and $B$ of $X$, the subset $C$ of $X$ is an asymptotic cut between $A$ and $B$ in $X$ if and only if $C$ is asymptotically disjoint from both $A$ and $B$ and for each $H\in \mathcal{F}$ there exists some $K\in \mathcal{F}$ such that each $H$-chain in $X$ from an element of $A$ to an element of $B$ in $X$ has a member in $CK$.
\end{proposition}
\begin{proof}
It is straightforward.
\end{proof}
By using Proposition \ref{metcut} and Proposition \ref{cut1} it is evident that Definition \ref{cut} is a natural generalization of the well known definition of the asymptotic cut in finitely generated groups (\cite{asdim}).\\
We recall the following definition from \cite{me2}.
\begin{definition}\label{large}
Assume that $(X,\lambda)$ is an AS.R space. A subset $C$ of $X$ is called a \emph{large scale separator} between two asymptotically disjoint subsets $A,B\subseteq X$, if it is asymptotically disjoint from both $A$ and $B$ and we have $X=X_{1}\cup X_{2}$ such that $X_{1}$ and $X_{2}$ are asymptotically disjoint from $A$ and $B$ respectively and if $L_{1}\lambda L_{2}$ for $L_{1}\subseteq X_{1}$ and $L_{2}\subseteq X_{2}$, then there is a subset $L$ of $C$ such that $L_{1}\lambda L$.
\end{definition}
We recall that a closed subset $C$ of a topological space $X$ is called a separator between closed and disjoint subsets $A$ and $B$ of $X$ is $X\setminus C=U\cup V$, such that $U$ and $V$ are two open and disjoint subsets of $X$ and they contain $A$ and $B$ respectively.\\ The following proposition is a direct consequence of Proposition 3.2 of \cite{me2}.
\begin{proposition}\label{asto}
Let $\mathcal{F}$ be a generating family on the normal topological group $G$. Assume that $\lambda_{\mathcal{F}}$ is an asymptotically normal AS.R on $G$ and $\mathcal{E}_{\mathcal{F}}$ is a proper and compatible coarse structure on $G$. Then for two asymptotically disjoint subsets $A$ and $B$ of $G$ if $C$ is a large scale separator between $A$ and $B$ then $\nu C$ is a separator between $\nu A$ and $\nu B$ in $\nu G$.
\end{proposition}
\begin{corollary}
Let $G$ be a normal and locally compact topological group such that $\lambda_{\mathcal{K}}$ is asymptotically normal. Then if $C\subseteq G$ is a large scale separator between asymptotically disjoint subsets $A$ and $B$ of $G$, then $\nu C$ is a separator between $\nu A$ and $\nu B$ in $\nu G$.
\end{corollary}
\begin{proof}
It is a direct consequence of Corollary \ref{com} and Proposition \ref{asto}.
\end{proof}
\begin{proposition}\label{septocut}
Let $(X,\mathcal{E})$ be a coarse space. Assume that $A$ and $B$ are two asymptotically disjoint subsets of the AS.R space $(G,\lambda_{\mathcal{E}})$. Then each large scale separator between $A$ and $B$ in $(X,\lambda_{\mathcal{E}})$ is an asymptotic cut between $A$ and $B$ in $(X,\mathcal{E})$.
\end{proposition}
\begin{proof}
Suppose that $C$ is a large scale separator between $A$ and $B$ and $X=X_{1}\cup X_{2}$ such that $X_{1}$ and $X_{2}$ satisfy the properties of Definition \ref{large}. Since $A$ and $X_{1}$ are asymptotically disjoint so $A\cap X_{1}$ is bounded. Similarly $B\cap X_{2}$ is bounded. So, without loss of generality, we can assume that $A\subseteq X_{2}$ and $B\subseteq X_{1}$. Suppose that $E\in \mathcal{E}$. Let $L_{1}$ denote the set of all $x\in X_{1}$ such that for some $E$-chain $x_{0},...,x_{n}$ from $A$ to $B$ there exists some $i\in \{1,...,n\}$ such that $x=x_{i}$ and $x_{i-1}\in X_{2}$. Similarly let $L_{2}$ denotes the set of all $x\in X_{2}$ such that for some $E$-chain $x_{0},...,x_{n}$ joining $A$ and $B$ and some $i\in \{0,...,n-1\}$, $x=x_{i}$ and $x_{i+1}\in X_{1}$. Clearly $L_{1}\subseteq F(L_{2})$ and $L_{2}\subseteq F(L_{1})$, where $F=E\cup E^{-1}$. So $L_{1}\lambda_{\mathcal{E}}L_{2}$. Thus there exists some $L\subseteq C$ such that $L_{1}\lambda_{\mathcal{E}}L$. It means that $L_{1}\subseteq O(L)$ for some $O\in \mathcal{E}$, which shows each $E$-chain joining $A$ and $B$ meets $O(C)$.
\end{proof}
\begin{definition}
Let $G$ be a group and let $\mathcal{F}$ denote a generating family on $G$. We say $G$ is $H$-\emph{convex} for some $H\in \mathcal{F}$, if $H$ generates $G$ and for each $K\in \mathcal{F}$ there exists some $n\in \mathbb{N}$ such that $K\subseteq H^{n}$.
\end{definition}
\begin{example}
Let $G$ be a finitely generated group and let $\mathcal{F}$ denote the family of all finite subsets of $G$. Assume that the finite subset $K$ generates $G$. Then clearly $G$ is $K$-convex. Recall that the coarse structure associated to the generating family $\mathcal{F}$ is equal to the bounded coarse structure associated to the metric $d_{K}$.
\end{example}
\begin{proposition}
Let $G$ be a compactly generated locally compact group. Consider $G$ with the AS.R $\lambda_{\mathcal{K}}$, where $\mathcal{K}$ denotes the family of all compact subsets of $G$. Then $G$ is $H$-convex for some compact subset $H\subseteq G$.
\end{proposition}
\begin{proof}
Let $H$ be a compact set of generators for $G$. Clearly $G=\bigcup_{i\in \mathbb{N}}H^{i}$. By using the Baire category theorem there exists some $n\in \mathbb{N}$ such that $H^{n}$ has nonempty interior. Choose the open neighbourhood $U$ of the identity such that for some $a\in H^{n}$, $aU\subseteq H^{n}$. Assume that $K\subseteq G$ is a compact subset. So there exists $a_{1},...,a_{l}$ such that $K\subseteq \bigcup_{i=1}^{l}a_{i}U$. Let $m\in \mathbb{N}$ be such that $a_{1}a^{-1},...,a_{l}a^{-1}\in H^{m}$. Now for $x\in K$ choose $i\in \{1,...,l\}$ such that $x\in a_{i}U$. Thus $x\in a_{i}a^{-1}aU\subseteq H^{m+n}$. It shows $K\subseteq H^{m+n}$, which completes our proof.
\end{proof}
\begin{proposition}\label{cuttosep}
Assume that $G$ is a group and $\mathcal{F}$ is a generating set on $G$. If $G$ is $H$-convex for some $H\in \mathcal{F}$ then each asymptotic cut between asymptotic disjoint subsets of $G$ is a large scale separator between them.
\end{proposition}
\begin{proof}
First it is straightforward to show that if the group $G$ is $H$-convex for some $H$ in $\mathcal{F}$ then it is also $H\cup H^{-1}$-convex. Thus without loss of generality we can assume that $H=H^{-1}$. Let $A$ and $B$ be two asymptotically disjoint subsets of $X$ and let $C$ be an asymptotic cut between $A$ and $B$. By the definition of asymptotic cut there exists some $K\in \mathcal{F}$ such that each $H$-chain connecting $A$ and $B$ meets $CK$. Since $A$ and $B$ are asymptotically disjoint from $C$, $A\cap CK$ and $B\cap CK$ are bounded so we can assume, without loss of generality, $A,B\subseteq G\setminus CK$. Let $Y=G\setminus CK$. For $a\in Y$ assume that $[a]_{H}^{Y}$ denotes all points $b\in Y$ such that there exists a $H$-chain in $Y$ connecting $a$ and $b$. Let $X_{1}=\cup_{a\in A}[a]_{H}^{Y}$ and $X_{2}=X\setminus X_{1}$. Since each $H$-chain connecting $A$ and $B$ meets $CK$, $A\subseteq X_{1}$ and $B\subseteq X_{2}$. Suppose that $L_{1}\subseteq X_{1}$ and $L_{2}\subseteq X_{2}$ and $L_{1}\lambda_{\mathcal{F}}L_{2}$. So there is some $O\in \mathcal{F}$ such that $L_{1}\subseteq L_{2}O$ and $L_{2}\subseteq L_{1}O$. Since we assumed that $G$ is $H$-convex, there exists some $n\in \mathbb{N}$ such that $O\subseteq H^{n}$. So $L_{1}\subseteq L_{2}H^{n}$. It shows that for each $c\in L_{1}$ there exists some $d\in L_{2}$ and a $H$-chain $x_{0}=c,...,x_{n}=d$. Since $L_{2}\subseteq X_{2}$ and $[c]_{H}^{Y}\subseteq X_{1}$, $d\notin [c]_{H}^{Y}$. Thus there is some $i\in \{1,..,n\}$ such that $x_{i}\in CK$. Thus there is some $k\in K$ such that $x_{i}k^{-1}\in C$. Let $\tilde{H}=H\cup ...\cup H^{n}$ and let $F=\tilde{H}(K^{-1}\cup K)$. Clearly $F=F^{-1}$. Since $x_{i}\in c\tilde{H}$ so $x_{i}k^{-1}\in cF$. Thus for each $c\in L_{1}$ we found some $g\in C$ such that $c\in gF$. It shows that $L_{1}\subseteq CF$. If we assume $L=C\cap L_{1}F$. Clearly $L_{1}\subseteq LF$ and $L\subseteq L_{1}F$, which shows $L\lambda_{\mathcal{F}}L_{1}$.
\end{proof}
\begin{remark}
One can shows that if $G$ is a group and $\mathcal{F}$ is a generating family on $G$ such that $G$ is $H$-convex for some $H\in \mathcal{F}$, then the AS.R space $(G,\lambda_{\mathcal{F}})$ is metrizable and it is $r$-convex for some $r>0$. Thus Proposition \ref{cuttosep} is a consequence of Proposition 4.9 of \cite{me2}.
\end{remark}
The following example shows that the inverse of Proposition \ref{septocut} is not true in general and in addition it shows being $H$-convex is an important condition for Proposition \ref{cuttosep}.
\begin{example}
Assume that $G$ is the additive group $\mathbb{Q}\times \mathbb{Q}$ and consider $G$ with the discrete topology. We assume $G$ with the AS.R $\lambda_{\mathcal{K}}$. Since each compact subset of $G$ is finite and $G$ is not finitely generated so $G$ is not $H$-convex for each compact subset $H$ of $G$. Remark that $G$ is hemicompact. Now suppose that $(p_{n})_{n\in \mathbb{N}}$ denotes the sequence of all prime numbers and $p_{1}<p_{2}<...$. Let $A=\{(n+\frac{1}{p_{n}},0)\mid n\in \mathbb{N}\}$ and $B=\{(0,n)\mid n\in \mathbb{N}\}$. It is easy to verify that for each finite subset $F$ of $G$, $(A+F)\bigcap B$ is finite. So $A$ and $B$ are asymptotically disjoint. Assume $F=\{(a_{1},b_{1}),...,(a_{k},b_{k})\}$ is a finite subset of $G$. For each $i\in \{1,..,k\}$ choose $m_{i}\in \mathbb{Z}$ and $n_{i}\in \mathbb{N}$ such that $a_{i}=\frac{m_{i}}{n_{i}}$. Suppose that $r\in \mathbb{N}$ is such that $p_{r}$ is strictly greater than each prime factor of $n_{1},...,n_{k}$. For $l_{1},...,l_{k}\in \mathbb{N}$, if $n+\frac{1}{p_{r}}+l_{1}a_{1}+...+l_{k}a_{k}=0$ then
$$np_{r}n_{1}...n_{k}+n_{1}...n_{k}+l_{1}m_{1}p_{r}n_{2}...n_{k}+l_{2}m_{2}p_{r}n_{1}n_{3}...n_{k}+...+l_{k}m_{k}p_{r}n_{1}...n_{k-1}=0$$
The last equality shows that there are integers $\alpha_{0},...,\alpha_{k}$ such that
$$\alpha_{0}p_{r}+n_{1}...n_{k}+\alpha_{1}p_{r}+...+\alpha_{k}p_{r}=0$$
So
$$-p_{r}(\alpha_{0}+...+\alpha_{k})=n_{1}...n_{k}$$
This means $p_{r}$ is a prime factor of $n_{1}...n_{k}$ which contradicts our chose of $p_{r}$. This argument shows that the set of all $i\in \mathbb{N}$ such that there exists a $F$-chain from $(0,i+\frac{1}{p_{i}})$ to an element of $B$ is finite. Therefore the set $C=\{(0,0)\}$ is an asymptotic cut between $A$ and $B$ ($C$ can be any finite subset of $G$). We are going to show that $C$ is not a large scale separator between $A$ and $B$. Suppose the set $X_{1}\subseteq X$ is such that $X_{2}=X\setminus X_{1}$ is asymptotically disjoint from $B$. So for each $\alpha \in \mathbb{Q}$ the set $A_{\alpha}=\{n\in \mathbb{N}\mid (\alpha,n)\in X_{2}\}$ is finite. For each $\alpha \in \mathbb{Q}$ if $A_{\alpha}\neq \emptyset$, let $M_{\alpha}=\max A_{\alpha}$. Thus $(\alpha,M_{\alpha})\in X_{2}$ and $(\alpha,M_{\alpha}+1)\in X_{1}$. If we assume that $X_{2}$ is asymptotically disjoint from $X_{1}$, then we find out there exists some $N\in \mathbb{N}$ such that for each $\alpha >N$, $A_{\alpha}=\emptyset$. It shows that $L_{1}=\{(n+\frac{1}{p_{n}},1)\mid n\geq N\}\subseteq X_{1}$. Let $L_{2}=\{(n+\frac{1}{p_{n}},0)\mid n\geq N\}$. Clearly $L_{2}\subseteq A$ and $L_{1}\lambda L_{2}$. Since $L_{1}$ and $L_{2}$ are unbounded (they are infinite), $X_{1}$ and $A$ are not asymptotically disjoint. This argument shows that it is impossible to have $X=X_{1}\bigcup X_{2}$ such that $X_{1}$ and $X_{2}$ are asymptotically disjoint and they are asymptotically disjoint from $A$ and $B$ respectively. Therefore $C$ can not be a large scale separator between $A$ and $B$.
\end{example}
Let us recall that if $Y$ is a nonempty subset of the coarse space $(X,\mathcal{E})$ we denote the induced coarse structure on $Y$ by $\mathcal{E}_{Y}$.
\begin{definition}\label{dimensiongrad}
Let $X$ be a set and let $\mathcal{E}$ be a coarse structure on $X$. We denote the \emph{asymptotic dimensiongrad} of $(X,\mathcal{E})$ by $\operatorname{asDg}_{\mathcal{E}}(X)$ and define it inductively as follows. We say $\operatorname{asDg}_{\mathcal{E}}(X)=-1$ if and only if $X$ is bounded. For $n\in \mathbb{N}$ we say $\operatorname{asDg}_{\mathcal{E}}(X)\leq n$ if for each two asymptotically disjoint subsets $A$ and $B$ of $(X,\lambda_{\mathcal{E}})$ there exists an asymptotic cut $C$ between $A$ and $B$ such that $\operatorname{asDg}_{\mathcal{E}_{C}}(C)\leq n-1$. For $n\in \mathbb{N}\bigcup \{0\}$ we say $\operatorname{asDg}_{\mathcal{E}}(X)=n$ if $\operatorname{asDg}_{\mathcal{E}}(X)\leq n$ and $\operatorname{asDg}_{\mathcal{E}}(X)\leq n-1$ is not true. We say $(X,\mathcal{E})$ has infinite asymptotic dimensiongrad if $\operatorname{asDg}_{\mathcal{E}}(X)\leq n$ is not true for each $n\in \mathbb{N}$.
\end{definition}
\begin{proposition}
Assume that  $(X,\mathcal{E})$ and $(Y,\mathcal{F})$ are two coarsely equivalent coarse spaces. Then $\operatorname{asDg}_{\mathcal{E}}X=\operatorname{asDg}_{\mathcal{F}}Y$.
\end{proposition}
\begin{proof}
Let us denote $\lambda_{\mathcal{E}}$ by $\lambda_{1}$ and $\lambda_{\mathcal{F}}$ by $\lambda_{2}$. We first show $\operatorname{asDg}_{\mathcal{E}}X\leq \operatorname{asDg}_{\mathcal{F}}Y$. We proceed by induction on $\operatorname{asDg}_{\mathcal{F}}Y$. If $\operatorname{asDg}_{\mathcal{F}}Y=-1$ then $Y$ is bounded and clearly $X$ is also bounded,  so $\operatorname{asDg}_{\mathcal{E}}X=-1$. Now assume that the inequality holds when $\operatorname{asDg_{\mathcal{F}}}Y\leq n$. Suppose that $\operatorname{asDg_{\mathcal{F}}}Y=n+1$. Also assume that $\phi: X\rightarrow Y$ is a coarse equivalence between coarse spaces $(X,\mathcal{E})$ and $(Y,\mathcal{F})$ and $\psi: Y\rightarrow X$ is a coarse inverse of $\phi$. Let $A$ and $B$ be two asymptotically disjoint subsets of $X$. By Proposition 2.3 of \cite{me3}, $\phi(A)$ and $\phi(B)$ are two asymptotically disjoint subsets of $Y$. Since $\operatorname{asDg}_{\mathcal{F}}Y=n+1$ there exists an asymptotic cut $C$ between $\phi(A)$ and $\phi(B)$ in $Y$ such that $\operatorname{asDg}_{\mathcal{F}_{C}}C\leq n$. It is easy to show that $(\psi(C),\mathcal{E}_{\psi(C)})$ is coarse equivalent to $(C,\mathcal{F}_{C})$, so by our assumption of induction $\operatorname{asDg}_{\mathcal{E}_{\psi(C)}}\psi(C)\leq n$. Now it remains to show that $\psi(C)$ is an asymptotic cut between $A$ and $B$. Again by using Proposition 2.3 of \cite{me3}, $\psi(C)$ is asymptotically disjoint from both $\psi(\phi(A))$ and $\psi(\phi(B))$. Since $\psi(\phi(A))\lambda_{1} A$ and $\psi(\phi(B))\lambda_{2} B$, $\psi(C)$ is asymptotically disjoint from both $A$ and $B$. Now assume that $E\in \mathcal{E}$. Since $\phi$ is a coarse map, $F=(\phi\times \phi)(E)\in \mathcal{F}$. Since $C$ is an asymptotic cut between $\phi(A)$ and $\phi(B)$, there exists some $O\in \mathcal{F}$ such that each $F$-chain joining an element of $\phi(A)$ to an element of $\phi(B)$ meets $O(C)$. Since $\psi$ is a coarse map, $L=(\psi\times \psi)(O)\in \mathcal{E}$. In addition since $\psi$ is a coarse equivalence and $\phi$ is its coarse inverse, there exists some symmetric $L^{\prime}\in \mathcal{E}$ such that for each $x\in X$, $(\psi(\phi(x)),x)\in L^{\prime}$. Let $M=L\circ L^{\prime}$. Assume that $a=x_{0},...,x_{n}=b$ is a $E$-chain from $a\in A$ to $b\in B$. So $\phi(x_{0}),...,\phi(x_{n})$ is a $F$-chain from $\phi(a)\in \phi(A)$ to $\phi(b)\in \phi(B)$ and hence it meets $O(C)$. It means there exists some $i\in \{0,...,n\}$ such that $(\phi(x_{i}),c)\in O$, for some $c\in C$. Thus $(\psi(\phi(x_{i})),\psi(c))\in L$ and hence $(x_{i},\psi(c))\in M$. It shows that each $E$-chain from $A$ to $B$ meets $M^{-1}(\psi(C))$ and hence $\psi(C)$ is an asymptotic cut between $A$ and $B$. Thus $\operatorname{asDg}_{\mathcal{E}}X\leq n+1=\operatorname{asDg}_{\mathcal{F}}Y$. Similarly one can show that $\operatorname{asDg}_{\mathcal{F}}Y\leq \operatorname{asDg}_{\mathcal{E}}X$.
\end{proof}
\begin{corollary}
Let $\mathcal{F}$ be a generating set on the group $G$ and let $\mathcal{I}$ be a generating set on the group $H$. If the coarse spaces $(G,\mathcal{E}_{\mathcal{F}})$ and $(H,\mathcal{E}_{\mathcal{I}})$ are coarsely equivalent then $\operatorname{asDg_{\mathcal{F}}}G=\operatorname{asDg_{\mathcal{I}}}H$.
\end{corollary}
It is worth mentioning that if we replace the word ``asymptotic cut" with the word ``large scale separator" in Definition \ref{dimensiongrad} then we have the definition of \emph{large scale inductive dimension}. Proposition \ref{septocut} shows that the asymptotic dimensiongrad is always less than or equal to the large scale inductive dimension.
\section{Set Theoretic Coupling}\label{sec4}
\begin{definition}\label{coup}
Let $G$ and $H$ be two groups. Assume that $\mathcal{F}$ and $\mathcal{L}$ are two generating sets on $G$ and $H$ respectively. We call the nonempty set $X$ a \emph{set theoretic coupling} for AS.R spaces $(G,\lambda_{\mathcal{F}})$ and $(H,\lambda_{\mathcal{L}})$ if $G$ and $H$ act on $X$ from right and left respectively, such that these actions commute with each other and there exists a subset $O\subseteq X$ such that the following properties hold,\\
i) $GO=OH=X$,\\
ii) For each $F\in \hat{\mathcal{F}}$ and $K\in \hat{\mathcal{L}}$,
$$F_{G}(K)=\{g\in G\mid gOK\cap OK\neq \emptyset\}\in \hat{\mathcal{F}}$$
$$F_{H}(F)=\{h\in H\mid FOh\cap FO\neq \emptyset\}\in \hat{\mathcal{L}}$$
iii) For each $F\in \mathcal{F}$ there exists some $E_{H}(F)\in \mathcal{L}$ such that
$$FO\subseteq OE_{H}(F)$$
and for each $K\in \mathcal{L}$ there exists some $E_{G}(K)\in \mathcal{F}$ such that
$$OK\subseteq E_{G}(K)O$$
\end{definition}
The following proposition shows that Definition \ref{coup} is equivalent to Definition 5.5.1 of \cite{loh} if we just deal with the finitely generated groups.
\begin{proposition}
Let $G$ and $H$ be two groups and let $\mathcal{F}$ and $\mathcal{F}^{\prime}$ denote the family of all finite subsets of $G$ and $H$ respectively. Suppose that $G$ and $H$ act from left and right on the set $X$ respectively and these actions commute. Then $X$ is a set theoretic coupling for $(G,\lambda_{\mathcal{F}})$ and $(H,\lambda_{\mathcal{F}^{\prime}})$ if and only if there exists a subset $O\subseteq X$ such that the following properties hold,\\
i) $GO=OH=X$,\\
ii) The subsets $F_{G}=\{g\in G\mid gO\cap O\neq \emptyset\}$ and $F_{H}=\{h\in H\mid Oh\cap O\neq \emptyset \}$ are finite subsets of $G$ and $H$ respectively.\\
iii) For all $g\in G$ and $h\in H$ there exist finite subsets $F_{H}(g)\subseteq H$ and $F_{G}(h)\subseteq G$ such that $gO\subseteq OF_{H}(g)$ and $Oh\subseteq F_{G}(h)O$.
\end{proposition}
\begin{proof}
The `only if' part of Proposition is a straightforward consequence of the fact $F_{G}=F_{G}(\{e_{H}\})$ and $F_{H}=F_{H}(\{e_{G}\})$, where $e_{G}$ and $e_{H}$ are neutral elements of $G$ and $H$ respectively. To prove the converse suppose that $O\subseteq X$ is such that $GO=OH=X$ and it satisfies the properties i), ii) and iii) of this proposition. Assume that $L\subseteq H$ is finite and $L=\{h_{1},...,h_{n}\}$. By the property iii) of this proposition there are finite subsets $F_{G}(h_{i})$ such that $Oh_{i}\subseteq F_{G}(h_{i})O$, for $i\in \{1,...,n\}$. Let $D=\cup_{i=1}^{n}F_{G}(h_{i})$. Now suppose that $g\in F_{G}(L)$. Since $OL\subseteq DO$, $gDO\cap DO\neq \emptyset$. It shows that there are $d_{1},d_{2}\in D$ and $o_{1},o_{2}\in O$ such that $gd_{1}o_{1}=d_{2}o_{2}$. So $d_{2}^{-1}gd_{1}O\cap O\neq \emptyset$, which shows $d_{2}^{-1}gd_{1}\in F_{G}$. Therefore $g\in DF_{G}D^{-1}$ and this shows $F_{G}(L)\subseteq DF_{G}D^{-1}$. Since $DF_{G}D^{-1}$ is finite, $F_{G}(L)$ is also finite. Similarly one can show for each finite subset $K$ of $G$, $F_{H}(K)$ is finite. The property iii) of Definition \ref{coup} is an immediate consequence of the property iii) of this proposition.
\end{proof}
\begin{proposition}\label{equiva}
Suppose that $G$ and $H$ are two groups and $\mathcal{F}$ and $\mathcal{L}$ are two generating families on $G$ and $H$ respectively. If $(G,\lambda_{\mathcal{F}})$ and $(H,\lambda_{\mathcal{L}})$ admit a set theoretic coupling then they are asymptotically equivalent.
\end{proposition}
\begin{proof}
Suppose that $X$ is a set theoretic coupling for $(G,\lambda_{\mathcal{F}})$ and $(H,\lambda_{\mathcal{L}})$. Our notation here have similar meanings to the notation in Definition \ref{coup}. Let $x\in O$ and define a map $f:G\rightarrow H$ such that for each $g\in G$, $g^{-1}x\in O(f(g))^{-1}$. Suppose that $D\in \hat{\mathcal{L}}$. If $g\in f^{-1}(D)$, $g^{-1}x\in OD^{-1}$. Let $L=D^{-1}\cup \{e_{H}\}$ ($e_{H}$ is the neutral element of $H$). So $x\in gOL$ and $x\in OL$. Thus $gOL\cap OL\neq \emptyset$, and it shows $g\in F_{G}(L)$. Therefore $f^{-1}(D)\subseteq F_{G}(L)$ and it shows $f^{-1}(D)$ is a bounded subset of $G$. Now suppose that $A,B\subseteq G$ and $A\lambda_{\mathcal{F}}B$. So there exists some $F\in \mathcal{F}$ such that $A\subseteq BF$ and $B\subseteq AF$. Without loss of generality we can assume $F=F^{-1}$. So for each $a\in A$ there exists some $b\in B$ such that $b^{-1}a\in F$ and $a^{-1}b\in F$. We have $a^{-1}xf(b)=a^{-1}bb^{-1}xf(b)\in a^{-1}bO\subseteq FO$. Besides $FO\subseteq OE_{H}(F)$ and $a^{-1}xf(b)\in O(f(a))^{-1}f(b)$. Let $K=E_{H}(F)\cup \{e_{H}\}$. So $a^{-1}xf(b)\in OK(f(a))^{-1}f(b)$ and $a^{-1}xf(b)\in OK$, therefore $(f(a))^{-1}f(b)\in F_{H}(K)$. If we let $K^{\prime}=F_{H}(K)\cup (F_{H}(K))^{-1}$ we have $f(A)\subseteq f(B)K^{\prime}$. Similarly one can show that $f(B)\subseteq f(A)K^{\prime}$, thus $f(A)\lambda_{\mathcal{L}}f(B)$. Thus $f$ is an AS.R mapping. For each $h\in H$ choose $\tilde{f}(h)\in G$ such that $xh\in \tilde{f}(h)O$. Similar argument can show that $\tilde{f}$ is an AS.R mapping. Now suppose that $A\subseteq G$. Let $a\in A$. We have $a^{-1}x\in O(f(a))^{-1}$ so $a^{-1}xf(a)\in O$. Besides $xf(a)\in \tilde{f}(f(a))O$ so $a^{-1}xf(a)\in a^{-1}(\tilde{f}(f(a)))O$. This shows that $a^{-1}\tilde{f}(f(a))\in F_{G}(e_{H})=P$. We showed $\tilde{f}(f(a))\in aP$ and $a\in \tilde{f}(f(a))P^{-1}$, for each $a\in A$. Thus $A\subseteq \tilde{f}(f(A))P^{-1}$ and $\tilde{f}(f(A))\subseteq AP$. We know from the definition of set theoretic coupling $M=P\cup P^{-1}\in \mathcal{F}$. Thus $A\lambda_{\mathcal{F}}\tilde{f}(f(A))$ for all $A\subseteq G$. One can show that $f(\tilde{f}(B))\lambda_{\mathcal{L}}B$ for all $B\subseteq H$ in similar way. Thus $f$ is an asymptotic equivalence and $\tilde{f}$ is its coarse inverse.
\end{proof}
Let $G$ be a (Hausdorff) topological group that acts on a locally compact topological space on the right (or on the left). This action is called to be continuous if the map $(g,x)\rightarrow gx$ from $G\times X$ to $X$ is continuous and it is called cocompact if there exists a compact subset $K$ of $X$ such that $GK=X$. The action of $G$ on $X$ is called proper if the set $\{g\in G\mid gD\bigcap D\neq \emptyset\}$ has compact closure for all compact subsets $D$ of $X$.
\begin{definition}
Suppose that $G$ and $H$ are two (Hausdorff) topological groups and $G$ and $H$ act continuously, properly and cocompactly on the locally compact topological space $X$ from left and right respectively. If the actions of $G$ and $H$ on $X$ commute, then $X$ is called a topological coupling for $G$ and $H$.
\end{definition}
\begin{proposition}
Let $G$ and $H$ be two (Hausdorff) topological groups and let $X$ be a locally compact topological space. Then if $X$ is a topological coupling for $G$ and $H$ then it is a set theoretic coupling for $(G,\lambda_{\mathcal{K}})$ and $(H,\lambda_{\mathcal{K}}^{\prime})$, where $\lambda_{\mathcal{K}}$ and $\lambda_{\mathcal{K}}^{\prime}$ denote the group compact AS.Rs on $G$ and $H$ respectively.
\end{proposition}
\begin{proof}
Since the actions of $G$ and $H$ are cocompact so clearly there exists a compact subset $K$ of $X$ such that $GK=KH=X$. For each point $x\in K$ suppose that $U_{x}$ denotes an open subset of $X$ with compact closure such that $x\in U_{x}$. Since $K$ is compact there are $x_{1},...,x_{n}\in K$ such that $K\subseteq \bigcup_{i=1}^{n}U_{x_{i}}$. So $W=\bigcup_{i=1}^{n}U_{x_{i}}$ is an open subset of $X$ such that $GW=WH=X$ and $\overline{W}$ is compact. Let $O=\overline{W}$. Clearly $O$ satisfies the property i) of \ref{coup}. Now suppose that $D\subseteq G$ is compact. Since these actions are continuous, $DO\subseteq X$ is compact and the properness of the action of $H$ leads to $F_{H}(D)$ is relatively compact. Similar argument shows that $F_{G}(D^{\prime})$ is relatively compact for each compact subset $D^{\prime}$ of $H$. This proves the property ii) of \ref{coup}. For proving the property iii) of \ref{coup} suppose that $D\subseteq G$ is compact. Since $WH=X$ and $DO$ is compact there are $h_{1},...,h_{n}\in H$ such that $DO\subseteq \bigcup_{i=1}^{n}Wh_{i}$. Clearly the set $E_{H}(D)=\{h_{1},...,h_{n}\}$ is the desired set for the property iii) of \ref{coup}. Similarly for each compact $D^{\prime}\subseteq H$ one can find a finite set $E_{G}(D^{\prime})$ such that $OD^{\prime}\subseteq E_{G}(D^{\prime})O$.
\end{proof}
\begin{example}\label{akhar}
Let $G$ be a locally compact Hausdorff topological group and suppose that $H$ and $K$ are two closed normal subgroups of $G$. If $G/H$ and $G/K$ are two compact topological groups then $G$ is a topological coupling for $H$ and $K$ with their group compact AS.Rs as topological groups. To see this it can be shown that there exists a compact subset $O$ of $G$ such that $\pi_{H}(O)=G/H$ and $\pi_{K}(O)=G/K$, where $\pi_{H}:G\rightarrow G/H$ and $\pi_{K}:G\rightarrow G/K$ are quotient maps (for example see Lemma 7.2.5 of \cite{dikran}). Clearly $HO=OK=G$, so the trivial actions of $H$ and $K$ on $G$ is cocompact. Let $D\subseteq G$ be a compact subset. Since $L=\{h\in H\mid hD\bigcap D\neq \emptyset\}=DD^{-1}\bigcap H$ and $DD^{-1}$ is a compact subset of $G$, $L$ has a compact closure and it shows that $H$ acts properly on $G$. Similarly one can shows that $K$ acts properly on $G$ too. Therefore Proposition \ref{equiva} shows that two topological groups $H$ and $K$ are asymptotically equivalent.
\end{example}
\begin{example}
Recall that a discrete subgroup $H$ of a locally compact group $G$ is called a \emph{uniform lattice} if there exists a compact subset $K$ of $G$ such that $HK=G$ ($KH=G$) (\S 5.5.1 of \cite{loh}). It is evident from Example \ref{akhar} and Proposition \ref{equiva} that two uniform lattices in a locally compact topological group are asymptotically equivalent.
\end{example}
\section*{ACKNOWLEDGMENTS}
The author wishes to express his gratitude to the anonymous referee for his (or her) several useful comments. The author was supported by Babol Noshirvani University of Technology with Grant Program No. BNUT/393072/99.

\end{document}